\documentclass[11pt]{amsart}
\usepackage{times}
\usepackage{graphicx}
\usepackage{amsmath,amsthm}
\usepackage{epstopdf}

\newtheorem{theorem}{Theorem}[section]

\newtheorem{proposition}[theorem]{Proposition}

\theoremstyle{definition}

\newtheorem*{definition*}{Definition}
\newtheorem*{example*}{Example}
\newtheorem*{remark*}{Remark}
\newtheorem*{question*}{Question}
\newtheorem*{problem*}{Problem}
\newtheorem*{note*}{Note}
\newtheorem*{claim*}{Claim}

\newcommand{\R}{\mathbf{R}}

\newcommand{\Z}{\mathbf{Z}}

\newcommand{\ra}{\rightarrow}

\newcommand{\tr}{\mathrm{tr\,}}

\title[Yang-Mills Connections]{Yang-Mills Connections on Surfaces and Representations of the Path Group}
\author{Kent E. Morrison}
\address{Department of Mathematics, California Polytechnic State University, San Luis Obispo, California 93407}
\curraddr{American Institute of Mathematics, 360 Portage Avenue, Palo Alto, California 94306}
\email{morrison@aimath.org}
\thanks{Proc. Amer. Math. Soc. \textbf{112} (1991) 1101--1106.}
\begin{document}
\large
\renewcommand{\baselinestretch}{1.2}   
\normalsize
\begin{abstract}We prove that Yang-Mills connections on a surface are characterized as those with the property that the holonomy around homotopic closed paths only depends on the oriented area between the paths. Using this we have an alternative proof for a theorem of Atiyah and Bott that the Yang-Mills connections on a compact orientable surface can be characterized by homomorphisms to the structure group from an
extension of the fundamental group of the surface. In addition for $M = S^2$ we obtain the results that the Yang-Mills connections on $S^2$ are isolated and correspond with the
conjugacy classes of closed geodesics through the identity in the structure group.
\end{abstract}
\maketitle
\vspace{1cm}

In 1954 Kobayashi [K] showed that connections on principal $G$-bundles over a manifold $M$ can be defined in terms of their parallel transport as homomorphisms from the group of closed paths of $M$ to the structure group $G$. More recently Atiyah and Bott [AB] showed that the Yang-Mills connections on a compact orientable surface can be characterized by homomorphisms to $G$ from an extension of the fundamental group of $M$. The purpose of this paper is to present a new proof of the result of Atiyah and Bott, using the path group formulation for connections. We show that Yang-Mills connections are characterized as those with the property that the holonomy around homotopic closed paths only depends on the oriented area between the paths. In addition for $M = S^2$ we easily obtain the result that the Yang-Mills connections on $S^2$ are isolated, a result obtained in [FH] by other means. We can see, too, that the equivalence classes of Yang-Mills connections on $S^2$ are in one- to-one correspondence with the conjugacy classes of closed geodesics of $G$ through the identity. This was described in the introduction of [AB] and worked out in detail in [FH].
\pagebreak

\section{The path group description of connections}
Fix a base point $x_0$ in $M$. The path group is defined with reference to the base point but up to isomorphism it is independent of the base point. The path group consists of equivalence classes of closed, piecewise smooth paths of $M$ where a closed path is the oriented image of a piecewise smooth, base point preserving map from $S^1$ to $M$. It may be thought of as the loop space with the operations of concatenation and reversal to make it into a group. A path may be reparameterized by choosing a different map with the same image as long as the orientation does not change. The group operation is concatenation so that $\lambda_1\lambda_2$ means ``traverse $\lambda_2$ and then traverse $\lambda_1$.'' The inverse of a path $\lambda$ is got by reversing the orientation. 
This means that $\lambda\lambda^{-1}$ must be identified with the constant path at
$x_0$ . It also follows that if $\lambda^*:[0,1]  \ra M$ is a parameterization of $\lambda$ with the property that there exist $0 \le t_1 \le t_2 \le t_3 \le 1$ such that $\lambda^*$ restricted to $[t_1,t_2]$ has exactly the same image as $\lambda^*$ restricted to $[t_2,t_3]$ but with opposite orientations, then $
\lambda$ is equal in the path group to the path parameterized by  
\[
[0,1] \ra M:t \mapsto 
       \begin{cases}\lambda^*(t),\quad  &0 \le t \le t_1, \\ 
                              \lambda^*(t_1), \quad &t_1 \le t \le t_3, \\
                              \lambda^*(t), \quad & t_3 \le t \le 1.
        \end{cases}
\]
The retraced piece of the path has been clipped.

Let $\Phi(M)$ denote the path group. It is a topological group with the compact-open topology. The fundamental group $\pi_1(M)$ is the group of path components of $\Phi(M)$. The main result of [K] is the following.

\begin{theorem} The equivalence classes of connections on principal $G$-bundles over $M$ are in one-to-one correspondence with the conjugacy classes of continuous homomorphisms from $\Phi(M)$ to $G$.
\end{theorem}

\begin{proof} This is an outline of the main steps. From a principal $G$-bundle with a connection the parallel transport around closed paths based at $x_0$ defines a homomorphism
$\rho:\Phi(M) \ra G$ called the holonomy. However, equivalent connections define holonomy maps which are conjugate ($\rho_1$ and $\rho_2$ are conjugate if there exists $g \in G$ so that
$\rho_1(\lambda) = g\rho_2(\lambda)g^{-1}$ for all $\lambda \in \Phi(M)$). In the other direction, given $\rho:\Phi(M) \ra G$ we construct a principal $G$-bundle $P$ and a connection as follows: let $S(M)$ be the space of open paths based at $x_0$, i.e., oriented images of piecewise smooth maps $\sigma :[0,1] \ra M$ with
$\sigma(0) = x_0$.  Now $\Phi(M)$ acts on the right on $S(M)$ since $\sigma\lambda$ is an open path based at $x_0$. (First traverse the closed path $\lambda$ and then $\sigma$.) Define an action of $\Phi(M)$ on $S(M) \times G$ by $(\sigma,g)\cdot \lambda = (\sigma\lambda,\rho(\lambda)^{-1}g)$. Define $P := (S(M) \times G) / \Phi(M)$ to be the orbit space of this action. In fact, $S(M)$ is a principal $\Phi(M)$-bundle over $M$, where the projection
$S(M) \ra M$ maps a path to its endpoint, and $P$ is just the principal $G$-bundle arising from the homomorphism of structure groups $\rho:\Phi(M) \ra G$. To describe the connection on $P$
we describe the horizontal lift of any curve $c:[0,1]\ra M$ with an initial point $p_0 \in P_{c(0)}$.
Let $c(0) = y_0$ , $c(1) = y_1$, and suppose $p_0$ is the equivalence class $[(\sigma_0,g)]$ where $\sigma_0$ is a path from $x_0$ to $y_0$ and $g \in G$. For each $t \in [0,1]$ define the path $\sigma_t$ to be $\sigma_0$ followed by the segment of $c$ between $c(0)$ and $c(t)$. Then the horizontal lift of the curve $c$ is $\tilde{c} : [0,1] \ra P : t \mapsto [(\sigma_t , g)]$. It is straightforward to check that $\tilde{c}$ is well-defined and that the holonomy around a closed path $\lambda$ is given by $\rho(\lambda)$.
\end{proof}

\section{Yang-Mills connections}
Put a Riemannian metric on $M$. Let $P \ra M$ be a principal $U(n)$-bundle. (More generally we could consider $G$-bundles for any compact Lie group $G$; the proofs that follow will work in that generality, but we will explicitly use $G = U(n)$ in this paper.) Let
$\langle X,Y \rangle:= \tr XY^*$ be the invariant inner product on the Lie algebra $\mathfrak{u}(n)$ of skew-Hermitian matrices.

These are the ingredients necessary to define the Yang-Mills functional on the space of unitary connections of $P \ra M$, namely the integral over $M$ of the norm squared of the curvature. The Yang-Mills connections are, by definition, the critical points of the Yang- Mills functional. Flat connections are zeros of the functional and are obviously the absolute minima on a bundle admitting flat connections.
The Euler-Lagrange equations asserting that the derivative of the Yang-Mills functional vanishes at a critical connection is the first of the Yang-Mills equations:
\[ \begin{cases}  d_A \,*\!F ( A ) = 0  \\ d_A \; F(A)=0 \end{cases}  \]
The second is the Bianchi identity that holds for all connections. The notation here is that of [AB]: $A$ is a connection, $d_A$ its covariant derivative, $F(A)$ the curvature of $A$, which is a section in $\Omega^2(M,\mathrm{ad} P)$, and $*$ is the Hodge star operator
\[  *:\Omega^k(M,\mathrm{ad\,} P) \ra \Omega^{n-k}(M,\mathrm{ad\,} P) \]
extended from scalar forms to $\mathrm{ad\,} P$-valued forms using the invariant inner product of $\mathfrak{u}(n)$.

Define the subgroup $\Phi_\omega(M) \subset \Phi(M)$ to consist of the contractible paths enclosing area zero. More precisely, $\sigma$ is in $\Phi_\omega(M)$ when $\int_S \omega = 0$ where $S$ is the interior of $\sigma$. The quotient group $\Phi(M) / \Phi_\omega(M)$ is the group of equivalence classes of closed paths; two
paths are equivalent if they are homotopic and the area between them is zero. The main result of this paper is the following characterization of Yang-Mills connections by their holonomy.

\begin{theorem}The gauge equivalence classes of Yang-Mills connections on all principal $U(n)$-bundles over $M$ are in one-to-one correspondence with the conjugacy classes of
homomorphisms from $\Phi(M) / \Phi_\omega(M)$ to $U(n)$.
\end{theorem}   

\begin{proof} First we prove that for an irreducible connection the following two statements are equivalent:
\begin{enumerate}
\item$A$ is a Yang-Mills connection on a $U(n)$-bundle. 
\item $F(A) = i\lambda I_n \otimes \omega$ for some $\lambda \in \R$.
\end{enumerate}

(1) $\Rightarrow$ (2)\; Consider an irreducible connection $A$ satisfying $d_A\, *F(A) = 0$, with holonomy representation $\rho: \Phi(M) \ra U(n)$, also irreducible. Then $*F(A) \in \Omega^0(M, \mathrm{ad\,} P)$ is an infinitesimal automorphism of the connection $A$; that means it is in the Lie algebra of the automorphism group of $A$, but that is just the subgroup of $U(n)$ that stabilizes the representation $\rho$ under conjugation. (More precisely the two groups are identified because a covariant constant section is determined by its value at a single point of $M$.) Since $\rho$ is irreducible its stabilizer is the subgroup isomorphic to $U(1)$ consisting of the scalar multiples of the identity, the center of $U(n)$. This means $*F(A)$ has the form $i\lambda I_n$ for some $\lambda \in \R$. Assume that the area element $\omega$ has been normalized to have total area 1. Then $F(A) = i\lambda I_n \otimes \omega \in \Omega^2(M,\mathrm{ad\,}P)$.

(2) $\Rightarrow$ (1)\; $d_A\,*\!F(A)=d_A*\!(i\lambda I_n \otimes \omega) = d_A(i\lambda I_n)=0$.

A connection on a $U(n)$-bundle can be split into a direct sum of irreducible connections. The holonomy representation, the curvature and the Hodge star operator also split in the same way, so that a connection is Yang-Mills if and only if each of its irreducible components is Yang-Mills. Next we show the equivalence of the following three statements for a connection $A$ on a $U(n)$-bundle.
\begin{enumerate}
\setcounter{enumi}{2}
\item The holonomy of $A$ factors through $\Phi(M) / \Phi_\omega(M)$.
\item  $F(A) = \Lambda \otimes \omega$ for some $\Lambda \in \mathfrak{u}(n)$.
\item $A$ is a Yang-Mills connection.
\end{enumerate} 

(3) $\Rightarrow$ (4)\;  Let $\Phi_0(M) \subset \Phi(M)$ be the subgroup of contractible paths. This subgroup is the connected component of the identity and contains the subgroup $\Phi_\omega(M)$. When the genus of $M$ is positive, the quotient group $\Phi_0(M) /\Phi_\omega(M)$ is isomorphic to $\R$, since the enclosed area characterizes each coset. When $M=S^2$ the quotient is isomorphic to $U(1)$. (See Theorem 2.3.) In either case the restricted holonomy is a homomorphism
$\overline{\rho}:\Phi_0(M) /\Phi_\omega(M) \ra U(n)$ and hence describes a one-parameter subgroup of $U(n)$. Let
$\Lambda \in \mathfrak{u}(n)$ be the infinitesimal generator of the one-parameter subgroup $\overline{\rho}$, so that
$\overline{\rho}(t) = \exp(t\Lambda)$.  For a matrix group $\Lambda = \lim_{t \ra 0}(\rho(t)-I)/t$. We will show that $F(A) = \Lambda \otimes \omega$.

At $x_0 \in M$ let $u, v$ be tangent vectors. Recall that $F(A)(u,v)$ is given by infinitesimal parallel translation around the rectangle spanned by $u$ and $v$. Choose coordinates $(x_1, x_2 )$ so that the area form $\omega = dx_1 \wedge dx_2$ in a neighborhood of $x_0$ (Darboux's Theorem). For each $t$ in some interval containing $0$, define $\sigma_t$ to be the closed parallelogram spanned by $tu$ and $tv$. Then $F(A)(u,v) = \lim_{t \ra 0} (\rho(\sigma_t ) - I)/t^2$ . Now $\rho(\sigma_t ) = \overline{\rho}(\omega(tu,tv)) = \overline{\rho}(t^2 \omega(u,v))$, since the area of the parallelogram spanned by $u$ and $v$ is $\omega(u,v)$. Making the substitution $s = t^2 \omega(u,v)$, the limit becomes
\[ \lim_{s \ra 0} \frac{\overline{\rho}(s) - I}{s}\omega(u,v) = \Lambda \omega(u,v).\] Therefore $F(A) = \Lambda \otimes \omega$.

(4) $\Rightarrow$ (5) \; $d_A\,*\!F(A)=d_A\,*\!(\Lambda \otimes \omega) =d_A(\Lambda)=0$.

(5) $\Rightarrow$ (3)\;  By splitting into irreducible components we may assume that $A$ is an irreducible Yang-Mills connection and hence that $F(A) = i\lambda I_n \otimes \omega$ for some $\lambda \in \R$ by the
equivalence of (1) and (2) above. The formula shows that the span of the curvature $\{F(A)(u,v) | u,v \in T_{x_{\raisebox{-2pt}{\tiny 0}}} M\}$ is the multiples of the identity, and hence by the Holonomy
Theorem [KN], the Lie algebra of the restricted holonomy group is isomorphic to $\R$. 
The holonomy around a contractible path $\sigma : [0,1] \ra M$ is $\exp( \int_0^1 \sigma^* A) = \exp(\int_\sigma A) = \exp(\int_S dA)$ by StokesÕ Theorem, where $\sigma$ is the boundary of $S$. But $dA$ is the curvature of $A$ since the holonomy around $\sigma$ is $\exp(\int_S F(A)) = \exp(i \lambda \,\mathrm{area}(S))\cdot I_n$ . If $\sigma \in \Phi_\omega(M)$, then $\mathrm{area}(S)=0$ and so $\sigma \in \mathrm{Ker}\, \rho$.
\end{proof}

Atiyah and Bott have shown that the Yang-Mills connections on a Riemann surface $M$ can be characterized as coming from representations of a central extension of $\pi_1(M)$. A brief summary of what is contained in \S 6 of [AB] is the following.
Let $M$ be a compact orientable surface of genus $g \ge 1$ endowed with a Riemannian metric (the complex structure of the Riemann surface is not used). The fundamental group
$\pi_1(M)$ has the standard presentation using $2g$ generators $\alpha_1,\dots,\alpha_g,\beta_1,\dots,\beta_g$ and one relation $\prod_{i=1}^g[\alpha_i,\beta_i]=1$. Define the central extension of $\pi_1(M)$ by $\Z$ using an additional generator $J$ that commutes with the generators $\alpha_i$, $\beta_i$ and satisfies the relation
$\prod_{i=1}^g[\alpha_i,\beta_i]=J$.
The subgroup generated by $J$ is isomorphic to $\Z$ and is the normal
subgroup of this extension, which is denoted by $\Gamma$. Thus we have the exact sequence
\[ 0 \ra Z \ra \Gamma \ra \pi_1(M) \ra 0. \]
Now extend the center of $\Gamma$ from $\Z$ to $\R$ to obtain the group $\Gamma_{\R}$ that fits into the exact sequence
\[ 0\ra \R\ra \Gamma_{\R} \ra \pi_1(M) \ra 0.\]

\begin{theorem}$\,$ If $M$ is a surface of genus $g \ge 1$, then the groups  $\;\Gamma_{\R}\,$ and $\;\Phi(M)/\Phi_\omega(M)\;$ are isomorphic.
\end{theorem}

\begin{proof} The natural projection $\Phi(M)/\Phi_\omega(M) \ra \pi_1(M)$ has kernel $\Phi_0(M)/\Phi_\omega(M)$ isomorphic to $\R$. Viewing $M$ as being constructed from the identification of the edges of a $4g$-sided polygon, with the edges labeled by the generators according to the relation, we see that the contractible path $\prod_{i=1}^g[\alpha_i,\beta_i]$ encloses the total area of $M$, normalized to be 1, and
so its class in $\Phi_0(M)/\Phi_\omega(M)$ is identified with $1 \in \R$. Thus $\Phi(M)/\Phi_\omega(M)$ is an
extension of $\pi_1(M)$ by $\R$ in exactly the same way that $\Gamma_{\R}$ is constructed by Atiyah and Bott.
\end{proof}

\begin{theorem} (Genus zero) The group $\Phi(S^2 )/\Phi_\omega(S^2)$ is isomorphic to $U(1)$. 
\end{theorem}

\begin{proof} On the sphere all paths are contractible so that $\Phi_0(S^2 ) = \Phi(S^2)$. The interior of the path $\sigma$ is the exterior of $\sigma ^{-1}$ and so the area enclosed by $\sigma\sigma ^{-1}$, which is the total area of 1, must also be considered as area 0. Therefore the quotient group $\Phi(S^2)/\Phi_\omega(S^2)$ is not isomorphic to $\R$ as it is for surfaces of positive genus, but instead is isomorphic to
$\R/\Z$ or $U(1)$.
\end{proof}

From this we recover the results of [FH] that the gauge equivalence classes Yang-Mills connections on $S^2$ are in one-to-one correspondence with the conjugacy classes of
homomorphisms from $U(1)$ to $U(n)$, which are the closed geodesics through the identity of $U(n)$, and that they are isolated. These results also hold for any compact Lie group $G$ in
place of $U(n)$.

\begin{proposition} Yang-Mills connections on $S^2$ are isolated. More precisely, the space $\mathrm{Hom}(U(1),G) / G $ of equivalence classes of Yang-Mills connections is discrete.
\end{proposition}

\begin{proof}  If $H$ and $G$ are Lie groups, then the Zariski tangent space at the equivalence class 
$[\rho]$ $\in$ $\mathrm{Hom}(H,G) / G$ can be identified with $H^1(H,\mathfrak{g})$ where $\mathfrak{g}$ is an $H$-module via $\mathrm{Ad} \circ \rho$ and $\mathrm{Ad}$ is the adjoint action of $G$ on $\mathfrak{g}$. Now when $H$ is compact its cohomology groups vanish, and $\mathrm{Hom}(H,G) / G$ has zero-dimensional tangent spaces. Notice that it is the compactness of $U(1)$ that matters and not the compactness of $G$.
\end{proof}


\begin{thebibliography}{999}
\bibitem[AB]{AB} M. F. Atiyah and R. Bott, \emph{The Yang-Mills equations over Riemann surfaces.} Philos. Trans. Roy. Soc. London Ser. A. \textbf{308} (1982), 523--615.
\bibitem [K]{K} S. Kobayashi, \emph{La connexion des vari\'{e}t\'{e}s fibr\'{e}es I and II.} C. R. Acad. Sci. Paris. \textbf{238} (1954) 318--319, 443--444.
\bibitem [KN]{KN} S. Kobayashi and K. Nomizu, \emph{Foundations of Differential Geometry}, vol. 1. New York; Wiley Interscience, 1963.
\bibitem [FH]{FH} Th. Friedrich and L. Habermann, \emph{Yang-Mills equations on the two-dimensional sphere}. Comm. Math. Phys. \textbf{100} (1985) 231--243.
\end{thebibliography}
\end{document}